\documentclass[10pt]{amsart}
\usepackage[utf8]{inputenc}

\usepackage{amsthm}
\usepackage{amsmath, amssymb}

\usepackage{enumitem}
\usepackage[all]{xy}

\theoremstyle{plain}
\newtheorem{theorem}{Theorem}[section]
\newtheorem{corollary}[theorem]{Corollary}
\newtheorem{definition}[theorem]{Definition}
\newtheorem{lemma}[theorem]{Lemma}
\newtheorem{proposition}[theorem]{Proposition}
\newtheorem{example}[theorem]{Example}

\theoremstyle{remark}

\DeclareMathOperator{\Aut}{Aut}

\DeclareMathOperator{\Part}{Part}
\DeclareMathOperator{\Age}{Age}

\DeclareMathOperator{\dom}{dom}
\DeclareMathOperator{\range}{range}
\DeclareMathOperator{\Forb}{Forb}

\usepackage{hyperref}
\usepackage{color, colortbl}
\hypersetup{pdfborder={0 0 0}}

\renewcommand{\restriction}{\mathord{\upharpoonright}}

\binoppenalty=\maxdimen
\relpenalty=\maxdimen

\title[Coherent extension and automorphism groups]{Coherent extension of partial automorphisms, free amalgamation and automorphism groups}

\author{Daoud Siniora}

\address{Department of Mathematics and Actuarial Science, The American University in Cairo}

\email{daoud.siniora@aucegypt.edu}

\author{S\l awomir Solecki}

\address{Department of Mathematics, Cornell University}

\email{ssolecki@cornell.edu}

\thanks{Research of D. Siniora supported by the University of Leeds.}  

\thanks{Research of S. Solecki supported by 
NSF grant DMS--1700426.}

\date{\today}

\begin{document}

\maketitle

\begin{abstract}
We give strengthened versions of the Herwig--Lascar and Hod\-kin\-son--Otto extension theorems for partial automorphisms of finite structures. Such 
strengthenings yield several combinatorial and group-theoretic consequences for homogeneous structures. For instance, 
we establish a coherent form of the extension property for partial automorphisms for certain Fra\"{i}ss\'{e} classes. We deduce from these results
that the isometry group of the rational Urysohn space, the automorphism group of the Fra\"{i}ss\'{e} limit of any Fra\"{i}ss\'{e} class that is the class of 
all $\mathcal{F}$-free structures (in the Herwig--Lascar sense), and the automorphism group of any free homogeneous structure over a finite relational 
language, all contain a dense locally finite subgroup. We also show that any free homogeneous structure admits ample generics.            
\end{abstract}

\section{Introduction}

The purpose of this paper is to prove results that strengthen 
the Herwig--Lascar extension theorem \cite[Theorem 3.2]{herwiglascar} and the Hodkinson--Otto extension theorem \cite[Theorem 9]{hodkinsonotto}. 
The strengthening is obtained by replacing the notion 
of {\em extension property for partial automorphisms} (EPPA) with the new notion of {\em coherent EPPA}. 
The strengthened versions are stated in Theorems \ref{T:main} and \ref{HOstrengthenedtheorem} below. 
As we demonstrate in the paper, these sharper versions are of interest in applications, for instance, to the structure of isometry groups of metric spaces 
and automorphism groups of free homogeneous structures. The proofs of the sharper versions consist of reorganizing of and adding new ingredients to the original proofs from 
\cite{herwiglascar} and \cite{hodkinsonotto}; for example, the usage in the model theoretic context of ideas coming from ergodic theory (Mackey range from \cite{mackey}; see Lemma~\ref{L:spec} below) appears new. 

In this paper, $\mathcal L$ denotes a finite {\em relational} language. Actually,
it suffices to assume that arities of symbols in $\mathcal L$ is
bounded. By a structure we understand an $\mathcal
L$-structure. Let $A$ be an $\mathcal{L}$-structure. A \textit{partial automorphism} of $A$ is an $\mathcal{L}$-isomorphism $p:U\to V$ where $U, V$ are substructures of $A$. We denote by $ \Part(A)$ the set of all partial automorphisms of $A$.

\begin{definition}\rm \label{EPPAdefinition}
A class $\mathcal{C}$ of finite $\mathcal{L}$-structures has the \textit{extension property for partial automorphisms (EPPA)} if for every $A\in\mathcal{C}$, there exists $B\in \mathcal{C}$ containing $A$ as a substructure such that every partial automorphism of $A$ extends to an automorphism of $B$. Such extension is called an \textit{EPPA-extension}.
\end{definition}

We now introduce a stronger notion of EPPA, which we call coherent EPPA, where in addition to extending partial automorphisms, we also require that the composition of the extensions of any two partial automorphisms to be equal to the extension of their composition understood appropriately. The notion of 
coherence we define now will be at the heart of the matter.

Let $P$ be a family of partial bijections between subsets of a set $X$.
We call a triple $(p_1, p_2, q)\in P^3$ {\em coherent} if
\[
{\rm dom}(p_2) = {\rm dom }(q),\, {\rm range}(p_1) = {\rm range}(q),\,
{\rm range}(p_2) = {\rm dom}(p_1)
\]
and
\[
q = p_1\circ p_2.
\]
In the situations we will encounter, the set $X$ will be finite, and so one of the three conditions:
${\rm dom}(p_2) = {\rm dom }(q),\, {\rm range}(p_1) = {\rm range}(q),\,$ or ${\rm range}(p_2) = {\rm dom}(p_1)$ 
can be eliminated without changing the meaning of the notion of
coherence.

\begin{definition}\rm 
Let $P$ and $S$ be families of partial bijections between subsets
of $X$ and between subsets of $Z$, respectively. A function
$\phi\colon P\to S$ is {\em coherent} if for each coherent triple
$(p_1, p_2, q)\in P^3$, its image $(\phi(p_1), \phi(p_2), \phi(q))\in
S^3$ is coherent. 
\end{definition}

Coherence is a notion of homomorphism
between families of partial bijections. Note that the composition
of two coherent functions is coherent.  
If both $P$ and $S$ are groups of
permutations of $X$ and $Z$, respectively, then a coherent
function from $P$ to $S$ is a group homomorphism. In most, but not all,
situations we will encounter, $X$ will be a subset of $Z$ and
$\phi(p)$ will be an extension of $p\in P$.

Here is the strengthening of the EPPA from Definition~\ref{EPPAdefinition}. 

\begin{definition}\label{coherentEPPAdef}\rm
 A class $\mathcal{C}$ of finite $\mathcal{L}$-structures has \textit{coherent EPPA} if for every $A\in \mathcal{C}$, there exists $B\in\mathcal{C}$ and a coherent map $\phi:{\rm Part(A)}\to {\rm Aut}(B)$ such that $A\subseteq B$ and, for every $p \in {\rm Part(A)}$, $\phi(p)$ extends $p$.
\end{definition}

We call $B$ in Definition \ref{coherentEPPAdef} a coherent EPPA-extension of $A$. Also note that the restriction of the map $\phi$ in this 
definition to $\Aut(A)$ 
gives a group embedding $\phi:\Aut(A)\to \Aut(B)$.

Our first aim is to formulate a theorem, Theorem~\ref{forbclasses}, that produces, under appropriate assumptions, coherent EPPA. We introduce now notions 
needed to state this theorem.

Let $A, B$ be $\mathcal{L}$-structures. A \textit{homomorphism} from $A$ to $B$ is a map $h:A\to B$ such that for every relation symbol $R\in \mathcal{L}$ and 
tuple $\bar{a}\in A$, if $A\models R(\bar{a})$, then $B\models R(h(\bar{a}))$. An \textit{embedding} $h$ of $A$ into $B$ is an injective homomorphism such that  
if $B\models R(h(\bar{a}))$, then $A\models R(\bar{a})$.

\begin{definition}\label{forbembeddingdef}\rm
Let $\mathcal{F}$ be a family of $\mathcal{L}$-structures.
\begin{itemize}[noitemsep,topsep=0pt]
\item An $\mathcal{L}$-structure $A$ is called {\em $\mathcal F$-free under homomorphisms} ({\em embeddings}, respectively) if there is no homomorphism 
(embedding, respectively) from a structure in $\mathcal F$ to $A$. 

\item Denote by $\Forb_h(\mathcal{F})$ the class of all finite $\mathcal{L}$-structures which are $\mathcal{F}$-free under homomorphisms.

\item Denote by $\Forb_e(\mathcal{F})$ the class of all finite $\mathcal{L}$-structures which are $\mathcal{F}$-free under embeddings.
\end{itemize}
\end{definition}

Note that the classes $\Forb_h(\mathcal{F})$ and $\Forb_e(\mathcal{F})$ consists of finite structures, and 
$\Forb_h(\mathcal{F}) \subseteq \Forb_e(\mathcal{F})$. When we just say $\mathcal{F}$-free we mean 
$\mathcal F$-free under homomorphisms. The notion of $\mathcal F$-freeness comes from 
Herwig--Lascar \cite[p.1994]{herwiglascar}. 

\begin{definition}\rm 
A class $\mathcal{C}$ of finite $\mathcal{L}$-structures is called a \textit{Fra\"{i}ss\'{e} class} if it contains 
countably infinite isomorphism types, is closed under isomorphism, and has the hereditary property, the joint embedding property, and the amalgamation property.
\end{definition}

A relational $\mathcal{L}$-structure $M$ is \textit{homogeneous} if it is countable and every partial isomorphism between finite substructures of $M$ 
extends to a total automorphism. Fra\"{i}ss\'{e}'s Theorem states that countably infinite homogeneous $\mathcal{L}$-structures arise as Fra\"{i}ss\'{e} 
limits of Fra\"{i}ss\'{e} 
classes of finite $\mathcal{L}$-structures. 

We now recall the notion of a Gaifman graph and a Gaifman clique.

\begin{definition}\rm
Let $A$ be an $\mathcal{L}$-structure, where $\mathcal{L}$ is a relational language. The \textit{Gaifman graph} of $A$, 
denoted by ${\rm Gaif}(A)$, is the graph whose vertex set is the domain of $A$, and whose edge relation is defined as follows: two distinct vertices 
$u, v \in A$ are adjacent if and only if there is an $n$-ary relation $R\in\mathcal{L}$ and an $n$-tuple $(a_1,\ldots,a_n)\in A^n$ such that 
$u,v\in \{a_1,a_2,\ldots,a_n\}$ and $A\models R(a_1,a_2,\ldots,a_n)$. 

We say that $A$ is a \textit{Gaifman clique} if ${\rm Gaif}(A)$ is a clique.
\end{definition}

\begin{theorem}\label{forbclasses} Let $\mathcal{L}$ be a finite relational language, and $\mathcal{F}$ be a family of finite $\mathcal{L}$-structures.
\begin{enumerate}[label=(\roman*), noitemsep,topsep=0pt]
\item If $\mathcal{F}$ is finite and $\Forb_h(\mathcal{F})$ is a Fra\"{i}ss\'{e} class, then $\Forb_h(\mathcal{F})$ has coherent EPPA.

\item If each element of $\mathcal{F}$ is a Gaifman clique, then $\Forb_e(\mathcal{F})$ has coherent EPPA. 
\end{enumerate}
\end{theorem}

The theorem above will follow from Theorems~\ref{T:main} and \ref{HOstrengthenedtheorem} stated below in the introduction. 

We first discuss several applications of Theorem \ref{forbclasses} to the automorphism groups of some Fra\"{i}ss\'{e} limits, 
and then give its proof later in this section. The automorphism group $\Aut(M)$ of a countably infinite structure $M$ is equipped 
with the \textit{pointwise convergence topology}. 
The basic open sets of this topology are the cosets of pointwise stabilisers of finite subsets of $M$. Thus, a subset $H \subseteq \Aut(M)$ is dense if 
for every $g\in \Aut(M)$ and every finite $A\subseteq M$ there is an $h\in H$ such that $g\restriction_A=h\restriction_A$. 
In Theorem \ref{coherentEPPADLFsubgrouptheorem} below, 
we show that coherent EPPA gives rise to the existence of a dense locally finite subgroup of the automorphism group of a Fra\"{i}ss\'{e} limit. Thus, 
Theorem~\ref{forbclasses}(i) together with Theorem \ref{coherentEPPADLFsubgrouptheorem} yield the following result.

\begin{corollary}\label{C:fr}
Let $\mathcal F$ be a finite family of finite $\mathcal{L}$-structures.  Assume $\Forb_h(\mathcal{F})$ is a Fra\"{i}ss\'{e} class with the Fra\"{i}ss\'{e} limit 
$M$. Then $\Aut({M})$ contains a dense locally finite subgroup.
\end{corollary}

We now discuss applications of the second part of Theorem \ref{forbclasses}. We first state the following observation, whose proof is given in 
Lemma~\ref{freeamalgamgaifmancliques} below.

\textbf{Fact.} \textit{A class $\mathcal{C}$ of a finite $\mathcal{L}$-structures is a free amalgamation class if and only if 
$\mathcal{C}=\Forb_e(\mathcal{F})$ for some family $\mathcal{F}$ of Gaifman cliques.}

This observation yields several interesting applications of Theorem \ref{forbclasses}(ii) to free amalgamation classes, and consequently to the automorphism 
group of a \textit{free homogeneous} relational structure. By a free homogeneous structure we understand the Fra\"{i}ss\'{e} limit of a free amalgamation 
class---see 
Section \ref{freeamalgamationsection} for the definitions. We collect the main results of this paper on free homogeneous structures in the following statement. 
See Theorems \ref{DLFsubgrouptheorem} and \ref{amplegenericstheorem} below for the details.

\begin{corollary}
The automorphism group of a free homogeneous $\mathcal{L}$-structure contains a dense locally finite subgroup, and has ample generics and the small index property.
\end{corollary}

The literature has interesting results on the automorphism group of free homogeneous structures; see, for example, \cite{ivanov}, \cite{macphersonthomas}, \cite{macphersontent}, and Macpherson's survey \cite{macphersonsurvey}.

Theorem~\ref{forbclasses}(i) follows from a more precise Theorem~\ref{T:main} below.

Theorem \ref{forbclasses} is based on the sharper statements of Herwig--Lascar Theorem and Hodkinson--Otto Theorem we obtain in this paper. We now state the strengthened version of the former theorem. 

\begin{theorem}\label{T:main}
Let $\mathcal F$ be a finite family of finite $\mathcal{L}$-structures. Let $A$ be a
finite $\mathcal F$-free (under homomorphisms) $\mathcal{L}$-structure. If there exists a (possibly infinite) $\mathcal F$-free $\mathcal{L}$-structure $M\supseteq A$ such that each element of $\Part(A)$ extends to an
automorphism of $M$, then there exists a finite $\mathcal{L}$-structure $B\supseteq A$ and a map $\phi: \Part(A) \to {\rm Aut}(B)$ such that 
\begin{enumerate}[noitemsep, topsep=0pt]
\item[(i)] $p \subseteq \phi(p)$ for each $p\in \Part(A)$;

\item[(ii)] $\phi$ is coherent; 

\item[(iii)] $B$ is $\mathcal F$-free. 
\end{enumerate}
\end{theorem}

The strengthening of Herwig--Lascar \cite[Theorem 3.2]{herwiglascar} consists of point (ii)
in Theorem~\ref{T:main} ensuring coherence of the extension
procedure. The structure $B$ which makes Theorem~\ref{T:main} true
is identical to the structure constructed in \cite{herwiglascar}; the
extensions $\phi(p)$ constructed in \cite{herwiglascar} are underdetermined;
by making additional choices in their definitions one forces the
extensions to fulfil conditions from the conclusion of
Theorem~\ref{T:main}. 

We next show how Theorem \ref{forbclasses}(i) follows from Theorem \ref{T:main}. Recall that the \textit{age} of a structure $M$ is the class of all finite substructures embeddable in $M$.

\textbf{Proof of Theorem \ref{forbclasses}(i).}
Let $\mathcal{L}$ be a finite relational language, and $\mathcal{F}$ be a finite family of finite $\mathcal{L}$-structures. Suppose that $\Forb_h(\mathcal{F})$ is an Fra\"{i}ss\'{e} class. Then by Fra\"{i}ss\'{e}'s Theorem there is a countably infinite homogeneous structure $M$ with $\Age(M)=\Forb_h(\mathcal{F})$. Thus, $M$ is $\mathcal{F}$-free under homomorphisms. Let $A\in \Forb_h(\mathcal{F})$, and view $A$ as a substructure of $M$. By homogeneity of $M$, every element of $\Part(A)$ extends to an automorphism of $M$. Therefore, by Theorem \ref{T:main} there exists a finite $B \in \Forb_h(\mathcal{F})$ with $A\subseteq B$, and a coherent map $\phi: \Part(A) \to {\rm Aut}(B)$ such that $\phi(p)$ extends $p$ for each $p\in \Part(A)$. So the class $\Forb_h(\mathcal{F})$ has coherent EPPA. \hfill $\blacksquare$

We now state the strengthened version of the Hodkinson--Otto extension theorem \cite[Theorem 9]{hodkinsonotto}. First, we introduce the following notion; see \cite[Definition 8]{hodkinsonotto}. 

\begin{definition}\rm\label{Gaifman clique faithful EPPA def}
An extension $B$ of a structure $A$ is called \textit{Gaifman clique faithful} if for every Gaifman clique $Q \subseteq B$, there is $g\in {\rm Aut}(B)$ such that $g(Q)\subseteq A$.
\end{definition}

\begin{theorem}\label{HOstrengthenedtheorem}
Let $A$ be a finite $\mathcal{L}$-structure. Then there exists a finite $\mathcal{L}$-structure $B\supseteq A$ and a map $\phi: \Part(A) \to \Aut(B)$ such that 
\begin{enumerate}[noitemsep, topsep=0pt]
\item[(i)] $p \subseteq \phi(p)$ for each $p\in \Part(A)$;

\item[(ii)] $\phi$ is coherent; 

\item[(iii)] $B$ is a Gaifman clique faithful extension of $A$. 
\end{enumerate}
\end{theorem}

Again the strengthening we add in the theorem above is the coherence of the extension procedure of partial automorphisms, and as before, the structure $B$ in Theorem \ref{HOstrengthenedtheorem} above is identical to the one constructed in \cite{hodkinsonotto}. 

We next show how Theorem \ref{forbclasses}(ii) follows from Theorem \ref{HOstrengthenedtheorem}.

\textbf{Proof of Theorem \ref{forbclasses}(ii).}
Let $\mathcal{F}$ be a family of Gaifman cliques. Take a structure $A\in \Forb_e(\mathcal{F})$, and consider the coherent EPPA-extension $B$ of $A$ guaranteed by Theorem \ref{HOstrengthenedtheorem}. We will show that $B \in \Forb_e(\mathcal{F})$. Suppose for the sake of contradiction that $B \notin \Forb_e(\mathcal{F})$, then there is some Gaifman clique $Q\in 
\mathcal{F}$ and an embedding $h:Q\to B$. By the Gaifman clique faithfulness of $B$, there is $g\in \Aut(B)$ such that $gh(Q)\subseteq A$. This means that the forbidden structure $Q$ embeds in $A$, contradicting $A\in \Forb_e(\mathcal{F})$. Thus, $B \in \Forb_e(\mathcal{F})$. \hfill $\blacksquare$

Theorem~\ref{T:main} has applications to metric spaces that are similar to 
the corollaries stated above. 
Replacing the role of \cite[Theorem 3.2]{herwiglascar} in the proof of \cite[Theorem 2.1]{soleckiisometries} by using Theorem \ref{T:main} instead, one obtains 
the following strengthening of \cite[Theorem 2.1]{soleckiisometries}.

\begin{corollary}
\label{C:metric}
Let $A$ be a finite metric space. There exists a finite metric
space $B$ such that $A\subseteq B$ as metric spaces, each partial
isometry $p$ of $A$ extends to an isometry $\phi(p)$ of $B$ and
the function $\phi$ is coherent.

Moreover, the distances between points in $B$ belong to the
additive semigroup generated by the distances between points in
$A$.
\end{corollary}

The conclusion of the theorem above gives that $\phi$ constructed
in this theorem when restricted to the isometry group of $A$,
${\rm Iso}(A)$, is a homomorphism, so necessarily an isomorphic
embedding,
\[
{\rm Iso}(A)\hookrightarrow {\rm Iso}(B).
\]
Furthermore, if the distances in $A$ are all rational numbers, then
so are the distances in $B$. Thus, Corollary~\ref{C:metric} yields that the class of all finite metric spaces over rational distances has coherent EPPA. 
The relational language here is 
\[
\mathcal{L}=\{R_r \mid r\in \mathbb{Q},\ r\geq 0\}
\]
where $R_r$ 
is a binary relation symbol. Given a metric space $A$ with a metric $d$ with rational distances, we view $A$ as an $\mathcal{L}$-structure by setting 
$A\models R_r(a,b)$ precisely when $d(a,b)=r$ for all $a,b\in A$.   
Keeping this formalization in mind, Corollary~\ref{C:ur} below follows from Corollary~\ref{C:metric}. It answers 
a question of Vershik, see \cite[6.13(5)]{kechrisrosendal}.

\begin{corollary}\label{C:ur}
The isometry group of the rational Urysohn space ${\mathbb U}_0$ contains a dense locally
finite subgroup.
\end{corollary}

After we proved Corollary~\ref{C:ur}, Rosendal gave an alternative proof of it in \cite[Theorem 10]{rosendal}.

The layout of the paper is as follows. In Section \ref{HLstrengthening} we prove Theorem \ref{T:main}.  In Section \ref{coherentGEPPAsection} we prove Theorem \ref{HOstrengthenedtheorem}. In Section \ref{freeamalgamationsection} we apply Theorem \ref{forbclasses}(ii) to free amalgamation classes. Finally, in Section \ref{automorphismgroupsection} we show that coherent EPPA is sufficient for the existence of a dense locally finite subgroup of the automorphism group of a Fra\"{i}ss\'{e} limit. We also deduce that free homogeneous structures admit ample generics.

The notions of coherent functions and coherent EPPA, the content of Section~\ref{HLstrengthening}, and Corollaries~\ref{C:metric} and \ref{C:ur} 
were found by the second author in the spring of 2007. 
They were written up and posted online in the fall of 2009. 
The content of 
Sections~\ref{coherentGEPPAsection}--\ref{freeamalgamationsection}
were part of the first author's PhD thesis \cite{sinioraPhD} from 2017.

\section{Strengthening of the Herwig-Lascar Extension Theorem}\label{HLstrengthening}

The proof of Theorem~\ref{T:main} is done in three stages. First, in Section~\ref{S:not} in
Proposition~\ref{P:nofr}, one shows that a finite structure $A$ can be
extended to a finite structure $B$ so that all partial
isomorphisms of $A$ extend to automorphisms of $B$ in a coherent
way. The $\mathcal F$-freeness condition is not involved. Then,
using Proposition~\ref{P:nofr}, in Section~\ref{S:str}, one shows in
Proposition~\ref{P:strong} that each finite {\em stretched}
structure $A$ that is $\mathcal F$-free (under homomorphisms), where $\mathcal F$
consists of {\em small} structures, can be extended to a finite
$\mathcal F$-free stretched structure so that each {\em strong}
partial isomorphism of $A$ extends to an automorphism of $B$ in a
coherent way. (All the terms mentioned in the preceding sentence
are defined in Section~\ref{S:str}). Finally, using
Proposition~\ref{P:strong}, one proves Theorem~\ref{T:main} in
Section~\ref{S:prm}. An important ingredient in this last proof is
a lemma that provides a construction of special structures. We
will give a new proof of this lemma based on an extension of the
ideas of Mackey \cite{mackey}.

\subsection{Two lemmas allowing the strengthening}

The following lemma will be used in the proofs of Propositions~\ref{P:nofr} and \ref{P:strong}. 
It is related to \cite[Lemma 4.11]{herwiglascar} and can be concatenated with that lemma to obtain its generalization. 
We will however only use the result below.

\begin{lemma}\label{L:simple}
Let $X$ be a finite set and let $P$ be a set of partial functions
from ${\mathcal P}(X)$ to ${\mathcal P}(X)$. Assume that for each
$p\in P$ there is $\sigma_p\in {\rm Sym}(X)$ such that for each
$a\in {\rm dom}(p)$ we have $p(a) = \sigma_p[a]$. Then there exists
$\phi\colon P\to {\rm Sym}(X)$ such that:
\begin{enumerate}[noitemsep]
\item[(i)] $p(a) = \phi(p)[a]$ for $a\in {\rm dom}(p)$ and

\item[(ii)] $\phi$ is coherent.
\end{enumerate}
\end{lemma}

\begin{proof} Each of the following two formulas extends each
$p\in P$ to a partial bijection $\widetilde{p}$:
\[
\widetilde{p}(X\setminus a) = \sigma_p[X\setminus a] \;\hbox{ where } a\in {\rm dom}(p)\] and
\[\widetilde{p}\big(\bigcap_{i=1}^k a_i\big) = \sigma_p[\bigcap_{i=1}^k
a_i] \;\hbox{ where } a_1,\ldots, a_k\in {\rm dom}(p).
\]
Of course, we let $\widetilde{p}$ be equal to $p$ on ${\rm
dom}(p)$. Note that if for $a\in {\rm dom}(p)$ it happens that
$X\setminus a\in {\rm dom}(p)$, then $\widetilde{p}(X\setminus a)
= p(X\setminus a)$. Similarly if for some $a_1, \dots , a_k\in
{\rm dom}(p)$ we have $\bigcap_{i=1}^k a_i\in {\rm dom}(p)$, then
$\widetilde{p}(\bigcap_{i=1}^ka_i) = p(\bigcap_{i=1}^ka_i)$. Thus,
$\widetilde{p}$ is indeed an extension of $p$. Additionally,
$\widetilde{p}$ is still induced by $\sigma_p$.

Since in the above formulas we have
\[
\widetilde{p}(X\setminus a) = X\setminus \sigma_p[a] = X\setminus
p(a)
\]
and
\[
\widetilde{p}\big(\bigcap_{i=1}^k a_i\big) = \bigcap_{i=1}^k
\sigma_p[a_i]= \bigcap_{i=1}^k p(a_i),
\]
one easily checks that if $p_1, p_2, q\in P$ are coherent, then so
are $\widetilde{p_1}, \widetilde{p_2}, \widetilde{q}$. By
iterating these extension operations we can suppose that the
domain and the range of each $p\in P$ is an algebra of subsets of
$X$. Moreover, $\sigma_p$ still induces $p$ on its domain.

Now fix a linear order of $X$. Let $p\in P$ and let $a$ be an atom
of the algebra that is the domain of $p$. Since $p(a) =
\sigma_p[a]$, we see that $a$ and $p(a)$ have the same number of
points. Define $\phi(p)\in {\rm Sym}(X)$ on points in $a$ to be
the only order preserving bijection from $a$ to $p(a)$. The
conclusion is easily verified by a calculation. In this verification it helps to
notice that if $a$ is an atom of ${\rm dom}(p)$, then $p(a)$ is an
atom of ${\rm range}(p)$.
\end{proof}

We recall from \cite[Definition 2.2]{herwiglascar} the definition of special extensions. 
We first fix some notation. Let $A$ be a structure and let $P\subseteq \Part(A)$. Let $W(P)$ be the set of all
words, including the empty word, in the alphabet $P\cup \{ p^{-1}\mid p\in P\}$. 
For a non-empty word $w\in W(P)$ with $w = p_1^{e_1}p_2^{e_2}\cdots p_n^{e_n}$ for
some $e_1, e_2, \dots, e_n\in \{ 1, -1\}$ and $p_1, p_2, \dots ,
p_n\in P$. 

Let ${\rm dom}(w)$
consist of all $x\in A$ such that $x$ is in the domain of $p_n^{e_n}$ and, 
for all $1\leq i<n$,  $p_{i+1}^{e_{i+1}}(p_{i+2}^{e_{i+2}}\cdots  (p_n^{e_n}(x)))$ 
is in the domain of $p_{i}^{e_{i}}$. For $x\in {\rm dom}(w)$, let  
\[
w(x) = p_1^{e_1}(p_2^{e_2}\cdots  (p_n^{e_n}(x)))
\]
Assume now that $B$
is an extension of $A$ and each $p\in P$ has an extension
$\phi(p)\in {\rm Aut}(B)$. We set
\[
\phi(w) = \phi(p_1)^{e_1}\circ \phi(p_2)^{e_2}\circ\cdots \circ
\phi(p_n)^{e_n}.
\]
For the empty word $w$, we let ${\rm dom}(w) = A$, $w(x)=x$, for each $x\in A$, and 
$\phi(w) = {\rm id}_B$. 
So for each $w\in W(P)$ and each $x\in {\rm dom}(w)$, we have
$\phi(w)(x) = w(x).$

Let $A, B, \phi$ be as above. We say that $B$ is a {\em special extension over $A$ and $\phi$}
if
\begin{enumerate}
\item[(i)] for each $y\in B$ there are $x\in A$ and a word $w\in
W(P)$ with $\phi(w)(x) = y$;

\item[(ii)] for all $y_1, \dots , y_r\in B$ with $R^B(y_1, \dots,
y_r)$ there are $x_1, \dots , x_r\in A$ and a word $w\in W(P)$
such that $y_i = \phi(w)(x_i)$ for $i\leq r$ and $R^A(x_1, \dots ,
x_r)$;

\item[(iii)] for $x_1, x_2\in A$, if $\phi(w)(x_1)=x_2$ for some
word $w\in W(P)$, then there is $v\in W(P)$ such that $\phi(v) =
\phi(w)$ and $x_1$ is in the domain of $v$; in particular, $v(x_1)
= x_2$.
\end{enumerate}

The lemma below is essentially \cite[Proposition 2.3]{herwiglascar}. We will
give a different proof of it that guarantees coherence. It is based on an extension of ideas
of Mackey \cite{mackey}.

\begin{lemma}\label{L:spec}
Let $A$ be a finite structure and let $P \subseteq \Part(A)$. Assume $B'$ is a finite extension of $A$ such
that each $p\in P$ has an extension $\psi(p)\in {\rm Aut}(B')$ with
$\psi$ being coherent. Then there exists a finite extension $B$ of
$A$ such that each $p\in P$ has an extension $\phi(p)\in {\rm
Aut}(B)$ such that $B$ is special over $A$ and $\phi$, and there is a 
homomorphism from $B$ to $B'$ equivariant with respect to $\phi(P)$, and $\phi$ is coherent. 
\end{lemma}

\begin{proof} Set $G= {\rm Aut}(B')$. This is a finite group. Let $1 = {\rm id}_{B'}$ be its identity element. 
Define the underlying set of $B$ as follows. Consider
\[
A\times G
\]
with the following relation on it
\[
(x,g)\sim (y,h) \Leftrightarrow \exists w\in W(P)\; (x\in {\rm
dom}(w)\hbox{ and } (w(x), \psi(w)\circ g) = (y,h)).
\]
One checks that $\sim$ is an equivalence relation and defines
\[
B= (A\times G)/\sim.
\]

For $p\in P$ and $[x,g]\in B$, let
\[
\phi(p)([x,g]) = [x, g\circ (\psi(p))^{-1}].
\]
It is easy to check that the operation above is well defined. Note
also that if $\psi$ is coherent, then so is $\phi$ since, for $p_1, p_2, q\in P$ with 
\[
{\rm dom}(p_2) = {\rm dom}(q),\, \range(p_1) = \range(q),\, \range(p_2)= {\rm dom}(p_1)
\]
the condition 
\[
\psi(q) = \psi(p_1)\circ \psi(p_2)
\]
immediately translates to
\[
\phi(q) = \phi(p_1)\circ \phi(p_2).
\]

Define now $\iota\colon A\to B$ by letting
\[
\iota(x) = [x,1].
\]
Note that $\iota$ is injective, since $[x,1]=[y,1]$ implies that
for some $w\in W(P)$ we have $x\in {\rm dom}(w)$, $y = w(x)$ and
$\psi(w)=1$. We get $w(x) = \psi(w)(x)=x$, so $y=x$.

We make $B$ into a structure by declaring that
\[
R^B(\phi(w)([x_1,1]), \dots , \phi(w)([x_r,1]))
\]
for some $w\in W(P)$ and $x_1, \dots , x_r\in A$ with $R^A(x_1,
\dots , x_r)$.

We observe that $\iota$ is an embedding from $A$ to $B$. To verify this observation, it suffices
to check that if $R^B([x_1, 1], \dots , [x_r,1])$ for some $x_1,
\dots , x_r\in A$, then $R^A(x_1, \dots , x_r)$. Assuming
$R^B([x_1, 1], \dots , [x_r,1])$, we can find $w\in W(P)$ and
$y_1, \dots , y_r\in A$ with $R^A(y_1, \dots , y_r)$ and with
\begin{equation}\notag
\begin{split}
[x_1,1] = \phi(w)([y_1,1])=\, &[y_1, \psi(w)^{-1}], \dots ,\\
&[x_r,1]= \phi(w)([y_r,1]) = [y_r, \psi(w)^{-1}].
\end{split}
\end{equation}
From this sequence of equalities we can find $w_i\in W(P)$ and
$x_i\in {\rm dom}(w_i)$, for $1\leq i\leq r$, so that
$w_i(x_i)=y_i$ and $\psi(w_i) = \psi(w)^{-1}$. It follows that
$\psi(w)^{-1}(x_i) = y_i$ so
\[
x_i=\psi(w)(y_i).
\]
Since $\psi(w)$ is an automorphism of $B'$ and, by
assumption, we have $R^A(y_1, \dots , y_r)$, we get $R^A(x_1, \dots
, x_r)$ as required.

Note that $\phi(p)$, for $p\in P$, is an extension of $p$ (if $A$
is viewed as a substructure of $B$ via $\iota$). Indeed, for $x\in
{\rm dom}(p)$, we get
\[
\phi(p)([x,1]) = [x, \psi(p)^{-1}] = [p(x), \psi(p)\circ
\psi(p)^{-1}] = [p(x), 1].
\]

We check now that $B$ is special over $A$ and $\phi$. It is clear
that the first two conditions in the definition of special
structure are fulfilled. To see the third condition, let $x_1,
x_2\in A$ and let $w\in W(P)$ be such that
\[
\phi(w)([x_1, 1]) = [x_2, 1].
\]
We need to find $v\in W(P)$ such that $x_1\in {\rm dom}(v)$,
$v(x_1)=x_2$, and $\phi(w) = \phi(v)$. Since
\[
[x_1, \psi(w)^{-1}] = [x_2, 1],
\]
there is $v\in W(P)$ such that $x_1\in {\rm dom}(v)$,
$v(x_1)=x_2$, and $\psi(v)\circ \psi(w)^{-1} = 1$, so $v$ is as
required.

To define a homomorphism from $B$ to $B'$ consider the
function $A\times G\to B'$ given by
\[
(x,g)\mapsto g^{-1}(x).
\]
Note that if $x\in {\rm dom}(w)$, then on the element $(w(x),
\psi(w)g)$ of the $\sim$-equivalence class of $(x,g)$, the
function above can be evaluated by
\[
(w(x), \psi(w)\circ g)\mapsto (\psi(w)\circ g)^{-1}(w(x)) =
g^{-1}(\psi(w)^{-1}(w(x))) = g^{-1}(x).
\]
It follows that this function induces a function $f$ from
$(A\times G)/\sim$ to $B'$, that is, we have
\[
f\colon B\to B'.
\]
To check that $f$ is a homomorphism assume that
\[
R^B(\phi(w)([x_1,1]), \dots , \phi(w)([x_r,1]))
\]
for some $w\in W(P)$ and some $x_1, \dots , x_r\in A$ with
$R^A(x_1, \dots , x_r)$. Note also that
\[
f(\phi(w)([x_i,1])) = f([x_i, \psi(w)^{-1}]) = \psi(w)(x_i).
\]
Since $R^A(x_1, \dots , x_r)$ and $\psi(w)$ is an automorphism of
$B'$, we get
\[
R^{B'}(\psi(w)(x_1), \dots , \psi(w)(x_r)),
\]
hence
\[
R^{B'}(f(\phi(w)([x_1,1])), \dots , f(\phi(w)([x_r,1]))),
\]
as required.
\end{proof}

\subsection{Extending isomorphisms without $\mathcal F$-freeness}\label{S:not}

\begin{proposition}\label{P:nofr}
Let $A$ be a finite structure. There exists a finite structure
$B\supseteq A$ and a coherent function $\phi: \Part(A) \to {\rm Aut}(B)$ with $p \subseteq \phi(p)$ for each $p\in \Part(A)$.
\end{proposition}

Proposition~\ref{P:nofr} states that the class of all finite $\mathcal{L}$-structures has coherent EPPA. It suffices to show the proposition above for 
$\mathcal L$ containing
only one relation symbol. Indeed, we just apply  the argument in \cite[Lemma 4.12]{herwiglascar}
extended by the observation that if $\phi_1$ and $\phi_2$ are coherent as functions from a
set of partial functions $P$ to ${\rm Sym}(X)$ and ${\rm Sym}(Y)$,
respectively, then so is the function $\phi\colon P\to {\rm
Sym}(X\times Y)$ given by
\[
\phi(p)(x,y) = (\phi_1(p)(x), \phi_2(p)(y)).
\]
Assume from this point on that $\mathcal L$ is a language with one
relation symbol. Moreover, assume that the arity of the only
symbol in $\mathcal L$ is $>1$. (The case of arity equal to $1$ is easy to
handle.) In this special case, we will get the proposition above from
\cite[Lemmas 4.8 and 4.9]{herwiglascar} that can be combined with each other
and with the sentence following \cite[Definition 4.6]{herwiglascar} to give
the following statement:

\noindent {\em Let $A$ be a finite structure. There is a finite
set $X$, a natural number $r$ and an $\mathcal L$-structure $B$
with the underlying set ${\mathcal P}(X)^r$ such that $A$ is a
substructure of $B$, for each $\sigma\in {\rm Sym}(X)$ the
bijection of ${\mathcal P}(X)^r$ induced by $\sigma$ as follows
\[
(a_1, \dots , a_r) \mapsto (\sigma[a_1], \dots , \sigma[a_r])
\]
is an automorphism of $B$, and each partial automorphism of $A$
extends to an automorphism of $B$ induced by some $\sigma\in {\rm
Sym}(X)$.}

So assume the statement above. For $p\in \Part(A)$,
let $D_p=\dom(p)$ and $R_p=\range(p)$. Let also
$\sigma_p\in {\rm Sym}(X)$ be such that the automorphism of $B$
induced by it extends $p$. For a set $E\subseteq B$, let
\[
\widetilde{E} = \{ a\in {\mathcal P}(X)\mid \exists (a_1, \dots , a_r)\in E\; \exists i\leq r\; a=a_i\}.
\]
Define a function $\widetilde{p}$ from $\widetilde{D_p}$ to
${\mathcal P}(X)$ by letting
\[
\widetilde{p}(a) = \sigma_p[a].
\]
A quick check shows that the range of $\widetilde{p}$ is
$\widetilde{R_p}$.

We claim that the function $p\mapsto \widetilde{p}$ where $p\in P$ is coherent. By the above computation of the range of 
$\widetilde{p}$, it suffices to show that if $p_1, p_2, q$ are partial isomorphisms of $A$ such that
\begin{equation}\label{E:ddd}
D_{p_2} = D_q,\, R_{p_1} = R_q,\, D_{p_1} = R_{p_2}, \hbox{ and }q = p_1\circ p_2,
\end{equation}
then
\[
\widetilde{D_{p_2}} = \widetilde{D_{q}},\, \widetilde{R_{p_1}} =
\widetilde{R_{q}},\, \widetilde{D_{p_1}} =
\widetilde{R_{p_2}}, \hbox{ and } \widetilde{q} = \widetilde{p_1}
\circ \widetilde{p_2}.
\]
The first three equalities follow immediately from the first three equalities of \eqref{E:ddd}. 
It remains to see the fourth one. Let $a\in \widetilde{D_{p_2}}$. So
for some $(a_1, \dots, a_r)\in D_{p_2}$, we have $a=a_i$ for some
$i\leq r$. For ease of notation assume $i=1$, so the tuple is $(a,
a_2, \dots , a_r)$. Then
\begin{equation}\notag
\begin{split}
p_2(a, a_2, \dots , a_r&) = (\sigma_{p_2}[a], \sigma_{p_2}[a_2], \dots , \sigma_{p_2}[a_r])\\ 
&\hbox{ and }\\
p_1(\sigma_{p_2}[a], \sigma_{p_2}[a_2], \dots , \sigma_{p_2}[a_r]&) = (\sigma_{p_1}[\sigma_{p_2}[a]],
\sigma_{p_1}[\sigma_{p_2}[a_1]], \dots ,
\sigma_{p_1}[\sigma_{p_2}[a_r]]).
\end{split}
\end{equation}
Therefore, we have 
\[
\begin{split}
(\sigma_q[a], \sigma_q[a_1], \dots , \sigma_q[a_r]) &= q(a, a_1, \dots, a_r)\\ 
&= (\sigma_{p_1}[\sigma_{p_2}[a]], \sigma_{p_1}[\sigma_{p_2}[a_1]], \dots , \sigma_{p_1}[\sigma_{p_2}[a_r]]),
\end{split}
\]
hence 
\[
\widetilde{q}(a) = \sigma_q[a] =\sigma_{p_1}[\sigma_{p_2}[a]] =
\widetilde{p_1}(\widetilde{p_2}(a)).
\]
Thus, $\widetilde{q} =\widetilde{p_1} \circ \widetilde{p_2}$. Now
apply Lemma~\ref{L:simple} to the family
$\{ \widetilde{p}\colon p\in \Part(A)\}$
to get a coherent assignment
$\widetilde{p}\mapsto \psi(\widetilde{p}) \in {\rm Sym}(X)$.
Then the assignment $p\mapsto \psi(\widetilde{p})$ is coherent as required.

\subsection{Extending strong isomorphisms of stretched 
structures preserving $\mathcal F$-freeness}\label{S:str}

We recall the notion of a stretched structure that comes from 
\cite[p. 2005]{herwiglascar}. 
Assume the relational language $\mathcal L$ contains distinguished unary
predicates $U_0, U_1, \dots , U_k$. We say that an $\mathcal{L}$-structure $A$ is
{\em stretched} if $U_0^A, U_1^A, \dots , U_k^A$ partition $A$ and
if $R^A(a_1,\ldots, a_r)$ for $a_i \in A$ and $R\in \mathcal{L}$, then $|\{a_1, \ldots, a_r\} \cap U_i^A| \leq 1$ for $1\leq i\leq k$. Note that this condition does not involve $U_0^A$.

Below in this section when we say a structure we mean a stretched
structure.

\textit{Fix a (of course, stretched) finite structure $A$.}

When we
say that $B$ is an {\em extension} we understand that $B$ is a
structure containing $A$ as a substructure. For $D\subseteq A$, we say that an extension $B$ is {\em based on} $D$ if no $b\in B\setminus A$ has 
links with elements of $A\setminus D$. In other words, $B$ is the free amalgam  (Definition \ref{freeamalgamdef}) of $A$ and $C$ over $C\cap A$ for a structure $C$ such that $C\cap A\subseteq D$. We write
\[
B= A*C
\]
A structure $C$ is called {\em small} if $U_i^{C}$ has at most one element, for each $0\leq i\leq k$
An extension $B$ is called a {\em short extension} if there is a small structure $C$ such that $B= A* C$.  
(Short extensions are defined in \cite[p.2007]{herwiglascar}.)

Let $p$ be a partial isomorphism of $A$ with domain $D\subseteq A$, and let $B$ be an extension based on $D$. We write 
\[
p(B)
\]
for the structure whose domain is $B$ and in which, for each relation symbol $R$ of arity $r$ and $b_1, \dots , b_r\in B$, we have $R^{p(B)}(b_1, \dots, b_r)$ 
precisely when 
\begin{itemize}
\item $b_1, \dots, b_r\in A$ and $R^B(b_1, \dots , b_r)$ or 

\item there exists $i_0$ with $b_{i_0}\not\in A$ and $R^B(c_1, \dots , c_r)$, where $c_i=b_i$, if $b_i\not\in A$ and $c_i= p(b_i)$, if $b_i\in A$ (so $b_i\in D$). 
\end{itemize}

For two extensions $B_1, B_2$, let
\[
B_1\leq B_2
\]
if there is a homomorphism from $B_1$ to $B_2$ equal to the identity function on $A$. 

Let $p$ be a partial isomorphism of $A$
with domain $D\subseteq A$. We say that $p$ is {\em strong} if for
each short extension $B$ based on $D$, $B\leq A$ if and only if
$p(B)\leq A$. (Strong partial isomorphisms are defined in
\cite[Definition 5.3(4)]{herwiglascar}.)

The remainder of this section consists of a proof of the following proposition. 

\begin{proposition}\label{P:strong}
Let $\mathcal F$ be a set of small structures. Assume that $A$ is
a finite $\mathcal F$-free structure. There exists a finite
$\mathcal F$-free extension $B$ of $A$ such that each strong $p\in \Part(A)$ has an extension $\phi(p) \in \Aut(B)$ and $\phi$ is a coherent function from 
the set of all strong partial isomorphisms of $A$ to ${\rm
Aut}(B)$.
\end{proposition}

We introduce notions and lemmas needed in the proof of the above theorem.  

\textit{For the remainder of this section, 
we adopt the convention that for a partial isomorphism $p$ of $A$, the domain and the range of $p$ are denoted by $D_p$ and $R_p$, respectively.}  

We say that an extension $B$ is a {\em strong extension} if for
each short extension $C$, $C\leq B$ implies $C\leq A$. (Strong
extensions are defined in \cite[Definition 5.3(3)]{herwiglascar}.)

A {\em pointed structure} $B$ is a structure with a distinguished point that is an element of $U^B_1$. 
A {\em pointed short extension} $B$ is a
short extension that is a pointed structure with the distinguished point that is not an element of $A$. 
Given two pointed extensions $B_1, B_2$ of $A$, we let
\[
B_1\leq^* B_2
\]
if there is a homomorphism $B_1\to B_2$ that is identity on
$A$ and maps the distinguished point of $B_1$ to the distinguished point of $B_2$.

A \textit{type} is a pair $t= (\Gamma, {\mathcal E})$ for which there is a
strong extension $B$ of $A$, $b\in U_1^B$,  and a subset $D$ of $A$ such that
\begin{itemize}
\item $\Gamma$ is a pointed structure whose domain is  $A\cup \{ *\}$, where $*$
is the distinguished point of $\Gamma$ not belonging to $A$, and, for $b_1', \dots, b_r'\in \Gamma$, $R^\Gamma(b_1', \dots , b_r')$
precisely when the following two conditions hold: \begin{enumerate}
    \item $R^B(b_1, \dots , b_r)$, where $b_i=b_i'$ if $b_i'\not=*$, and $b_i=b$ if $b_i'=*$,
    
    \item if $b_{i_0}'= *$ for some $i_0$, then $b_i\in D$ for all $i$ with $b_i'\not= *$. 
\end{enumerate}

\item $\mathcal E$ is the family of all pointed short extensions that
are maximal with respect to $\leq^*$ among all pointed short
extensions $C$ based on $D$ with $C\leq^* B$; we assume that
$\mathcal E$ does not contain two distinct isomorphic
structures.
\end{itemize}

We write $t_B(b/D)$ for the pair $(\Gamma, {\mathcal E})$ as above. 
(The above definition is a part of \cite[Definition 5.7]{herwiglascar}. Types are
defined in \cite[Definition 5.17]{herwiglascar}.)

To see that the above definition coincides with the one from \cite{herwiglascar}, one applies \cite[Lemma 5.18]{herwiglascar} and then argues as follows. 
Define a pointed short extension $C$ to be {\em irreducible} if the Gaifman graph of $C$ is connected on the set $C\setminus
A$. It is easy to see that if $B$ is a pointed extension and $C_0$ is a pointed short extension that is maximal with respect to
$\leq^*$ among all pointed short extensions $C$ with $C\leq^*B$, then $C_0$ is irreducible. 

We say that a type $t = (\Gamma, {\mathcal E})$ is {\em based on}
$D\subseteq A$ if each point of $A$ having a link with $*$ in
$\Gamma$ belongs to $D$ and for each $C\in {\mathcal E}$ is based
on $D$.

Given a partial isomorphism $p$ of $A$ whose domain is $D$ and a type $t= (\Gamma, {\mathcal E})$ based on $D_p$, let
\[
p(t) = (p(\Gamma), \{ p(C)\mid C\in {\mathcal E}\}).
\]

We note the following lemma, which implies that $p(t)$ is a type if $B$ is a strong extension and $p$ is a strong partial isomorphism. 

\begin{lemma}\label{L:inv} 
Let $B$ be a strong extension and let $p$ be a strong partial isomorphism of $A$. 
\begin{enumerate}
\item[(i)] If $t$ is a type based on $D$, then $p(t)$ is a type
and it is based on $R_p$.

\item[(ii)] If $b\in D_p$, then 
\[
t_B(p(b)/R_p) = p(t_B(b/D_p)).
\]
\end{enumerate}
\end{lemma}

Point (i) of the lemma above is \cite[Lemma 5.21]{herwiglascar}. 
Point (ii) is obvious from the definitions and strongness of $p$ and is
mentioned in the sentence preceding \cite[Lemma 5.2]{herwiglascar}.

The next lemma is, in a way, a converse of Lemma~\ref{L:inv}(ii). 

\begin{lemma}\label{L:down}
Let $B$ be a strong extension, and let $p$ be a strong partial
isomorphism of $A$. Let $q$ be a
partial function from $B$ to $B$ with domain $D_p\subseteq E\subseteq D_p\cup
U^B_1$. Assume that $q$ is an injection, that it extends $p$ and
that for each $b\in U^B_1\setminus D_p$
\[
t_B(q(b)/R_p) = p(t_B(b/D_p)).
\]
Then $q$ is a strong partial isomorphism of $B$.
\end{lemma}

The lemma above is \cite[Lemma 5.22]{herwiglascar} together with the
argument at the top of \cite[p. 2015]{herwiglascar}.

Finally, we have the following fundamental result.  

\begin{lemma}\label{L:fintype} There exists a strong extension $B$ of $A$ such that
for each strong partial isomorphism $p$ of $A$ and each type $t$
based on $D_p$ we have
\[
|\{ b\in U^B_1\mid t_B(b/D_p) = t\}| = |\{ b\in U^B_1\mid
t_B(b/R_p) = p(t)\}|.
\]
\end{lemma}

The proof of this lemma is given in \cite{herwiglascar} upper half of page 2017.

\begin{proof}[Proof of Proposition~\ref{P:strong}]
First, we show the following statement. 

\textit{There exists a strong extension $B'$ of $A$ such that each strong
partial isomorphism $p$ of $A$ admits an extension to a strong
partial isomorphism $\phi'(p)$ with domain $D_p\cup U^B_1$ so that
$\phi'$ is coherent.}

To justify this statement, let $B'$ be as in
Lemma~\ref{L:fintype}. Given a strong partial isomorphism $p$
of $A$, define a partial function $\widetilde{p}$ from ${\mathcal P}(U_1^{B'})$ to ${\mathcal P}(U_1^{B'})$ as follows. The domain of
$\widetilde{p}$ consists of sets of the form 
\[
\{ b\},\; \hbox{ for } b\in D_p\cap U_1^{B'},
\]
and
\[
\{ b\in U_1^{B'}\setminus D_p\mid t(b/D_p) = t\},\; \hbox{ for a type $t$ based on }D_p.
\]
Note that the domain of $\widetilde{p}$ is a partition of $U_1^{B'}$. Define $\widetilde{p}$ on the sets of the first kind by
\[
\widetilde{p}(\{ b\}) = \{ p(b)\}
\]
and on the sets of the second type by
\[
\widetilde{p}(\{ b\in U_1^{B'}\setminus D_p\mid t(b/D_p) = t\}) = \{
b\in U_1^{B'}\setminus D_p\mid t(b/R_p) = p(t)\}.
\]
Observe that $\widetilde{p}$ depends only on $p\restriction (D_p\cap U_1^{B'})$. 
The condition in the conclusion of Lemma~\ref{L:fintype}
together with Lemma~\ref{L:inv} ensure that there exists
$\sigma_p\in {\rm Sym}(U^{B'}_1)$ that induces $\widetilde{p}$ by
\begin{equation}\label{E:psa}
\widetilde{p}(a) = \sigma_p[a]
\end{equation}
for all $a$ in the domain of $\widetilde{p}$. Furthermore, it is clear that each permutation in ${\rm Sym}(U_1^{B'})$ inducing $\widetilde{p}$ as
in \eqref{E:psa} extends $p\restriction U_1^{B'}$. Now Lemma~\ref{L:simple} gives us a coherent extension map 
\[
p\restriction (D_p\cap U_1^{B'}) \to \sigma_p \in {\rm Sym}(U_1^{B'})
\]
with \eqref{E:psa}. 
Let $\phi'(p)$ be the joint extension of $p$ and $\sigma_p$ to $D_p\cup U_1^{B'}$. It is clear that the function $p\mapsto \phi'(p)$ is a coherent extension. 
By Lemma~\ref{L:down}, each $\phi'(p)$ is a strong partial isomorphism of $A$.

From the above statement we get the following consequence. 

\textit{There exists a strong extension $B''$ of $A$ such
that each strong partial isomorphism $p$ of $A$ admits an
extension to a strong partial isomorphism $\phi''(p)$ with domain
$D_p\cup (B''\setminus U^{B''}_0)$ so that $\phi''$ is coherent.}

This consequence is established by producing a sequence of strong extensions
\[
A=B_0\subseteq B_1\subseteq \cdots \subseteq B_k
\]
so that the statement above, applied to $U_{i+1}^{B_i}$ in place of
$U_1^B$, is used to obtain $B_{i+1}$ from $B_i$. We finally let
\[
B''= U_0^{B_0}\cup U_1^{B_1} \cup\cdots \cup U_k^{B_k},
\]
and it is easy to see that this $B''$ is as required.

Now, take the structure $B''$ and the coherent extension $\phi''$
constructed above. Apply
Proposition~\ref{P:nofr} to $B''$ to obtain an extension $B'''$ of $B''$ and
an extension $\phi'''(p)\in {\rm Aut}(B')$ of $\phi''(p)$ for each
strong partial isomorphism $p$ of $A$ so that $\phi'''$ is
coherent. Let $B$ be the structure generated by $A$ using all
$\phi'''(p)$ with $p$ a strong partial isomorphism of $A$. It is
easy to see that $B$ is stretched (with respect to the unary
predicates $U_0, \dots , U_k$). Define
\[
\phi(p) = \phi'''(p)\restriction B.
\]
The structure $B$ and the extension $\phi$ are as required.
\end{proof}

\subsection{Extending isomorphisms with $\mathcal F$-freeness}\label{S:prm}
In this section, we finish the proof of Theorem~\ref{T:main}. 

We have a fixed finite relational language $\mathcal L$. First we
claim that it suffices to prove Theorem~\ref{T:main} under the
assumption that $M$ and all structures in $\mathcal F$ are
irreflexive. (This argument comes from \cite{herwig}.) We call a
structure $N$ {\em irreflexive} if for each relation symbol $R$,
say of arity $r$, if, for some $x_1, \dots , x_r\in N$, we have
$R^N(x_1, \dots , x_r)$, then $x_{i_1}\not= x_{i_2}$ for $i_1\not=
i_2$.

There is a canonical way to change $\mathcal L$ to ${\mathcal L}'$
to make each $\mathcal L$-structure into an irreflexive ${\mathcal
L}'$-structure. Given $R\in {\mathcal L}$ of arity $r$ and a
partition $S$ of $\{ 1, \dots , r\}$ into $s$ pieces, let
${\mathcal L}'$ contain a relation symbol $R_S$ of arity $s$.
Given an $\mathcal L$-structure $N$, interpret $R_S$ in it as
follows: $R_S^N(y_1, \dots , y_s)$ precisely when $R^N(x_1, \dots
, x_r)$, where $x_i = y_j$ for $i$ in the $j$-th element of the
partition $S$. Also each ${\mathcal L}'$-structure can be, in a
canonical way, made into an $\mathcal L$-structure. (These two
processes are inverses of each other only when we go from
$\mathcal L$ to ${\mathcal L}'$ first and then back to $\mathcal
L$.)

Now we are given $\mathcal L$-structures $A$, $M$, with
$A\subseteq M$, a finite family of $\mathcal L$-structures
$\mathcal F$, and a set $P$ of partial isomorphisms of $A$. We
assume that $M$ is $\mathcal F$-free under homomorphisms. We can assume, and we do,
that $\mathcal F$ is closed under taking homomorphisms. We
make $A$, $M$, and all the structures in $\mathcal F$ into
${\mathcal L}'$-structures in the canonical way described above.
Note that $A$ is still a substructure of $M$, $M$ is still
$\mathcal F$-free and each element of $P$ is still a partial
isomorphism of $A$, but the structures $A$, $M$, and all
structures in $\mathcal F$ are now irreflexive. Assuming that we
have Theorem~\ref{T:main} for irreflexive structures in its
assumptions, we get an ${\mathcal L}'$-structure $B$ (not
necessarily irreflexive) as in the conclusion of this theorem. By
turning $B$ in the canonical fashion into an $\mathcal
L$-structure, it is easy to check that we get the conclusion of
the theorem for the $\mathcal L$-structure $A$ and the set $P$;
this checking uses the fact that $\mathcal F$ is closed under
taking homomorphisms.

Therefore, from this point on we assume that $M$ and all
structures in $\mathcal F$ are irreflexive already with respect to
$\mathcal L$.

Let $k$ be bigger than the largest arity of a relation in
$\mathcal L$ and than the size of each structure in $\mathcal F$.
Let ${\mathcal L}^+$ be $\mathcal L$ together with $k+1$ new unary
relation symbols $U_0, U_1, \dots , U_k$. Stretched structures
below are stretched with respect to these unary predicates.

With each $\mathcal L$-structure $B$ we associate a stretched
${\mathcal L}^+$-structure $\widehat{B}$ as follows. The
underlying set of $\widehat{B}$ is $B\times \{ 0, 1, \dots , k\}$.
We interpret $U^{\widehat{B}}_i$ as $B\times \{ i\}$ and for $R\in
{\mathcal L}$, we set
\[
R^{\widehat{B}}((b_1,i_1), \dots , (b_r, i_r))
\]
precisely when $R^B(b_1, \dots , b_r)$ and the $i_j$-s with
$i_j>0$ are distinct from each other. This makes $\widehat{B}$
into a stretched structure. Note that $B$ is isomorphic to the
reduct to $\mathcal L$ of the substructure of $\widehat{B}$ with
the underlying set $U_0^{\widehat{B}}$. To each partial
isomorphism $p$ of $B$ we associate a partial isomorphism
$\widehat{p}$ of $\widehat{B}$ by letting
\[
\widehat{p}(b,i) = (p(b), i)
\]
for $b$ taken from the domain of $p$. Note that the function $p\mapsto
\widehat{p}$ is coherent.

To show Theorem~\ref{T:main}, assume we are given a finite
irreflexive $\mathcal L$-structure $A$, a set $P$ of partial
isomorphisms of $A$ and an irreflexive $\mathcal F$-free
$\mathcal L$-structure $M$ containing $A$ with each $p\in P$ extending to an automorphism of
$M$. Let ${\mathcal F}^+$ consist of all stretched small ${\mathcal L}^+$-structures that
are expansions of structures in $\mathcal F$. Consider $\widehat{A}$, $\widehat{M}$ and $\widehat{P}
= \{ \widehat{p}\mid p\in P\}$. Note that elements of
$\widehat{P}$ are strong in $\widehat{M}$. It is now easy to find
a finite structure $A'$ with
\[
\widehat{A}\subseteq A'\subseteq \widehat{M}
\]
such that each element of $\widehat{P}$ is strong in $A'$. Note that
since $\widehat{M}$ is ${\mathcal F}^+$-free, so is $A'$.
Proposition~\ref{P:strong} allows us to find a stretched structure $B'$
that is ${\mathcal F}^+$-free and such that each element of
$\widehat{P}$ extends to $B'$ and the extension is coherent. Now,
using Lemma~\ref{L:spec}, we find a special extension $B$ of
$\widehat{A}$ such that all elements of $\widehat{P}$ extend to $B$
coherently and there is a homomorphism from $B$ to $B'$. Using
speciality of $B$ we show that $B$ is a stretched structure. It is
${\mathcal F}^+$-free since there is a homomorphism $B\to B'$.
Consider now the reduct to $\mathcal L$ of the substructure of $B$
with the underlying set $U_0^B$. One can prove that this structure
is $\mathcal F$-free (see the middle half of \cite[p. 2006]{herwiglascar};
this argument uses irreflexivity of the elements of $\mathcal F$)
and, easily, $A\subseteq U_0^B$. This is the desired structure.

\section {Strengthening of the Hodkinson-Otto Extension Theorem} \label{coherentGEPPAsection}

Hodkinson and Otto \cite{hodkinsonotto} proved a Gaifman clique constrained strengthening of EPPA building on the work of Herwig and Lascar. In this section, we show that the strengthened EPPA they proved can be made coherent, that is, we prove Theorem~\ref{HOstrengthenedtheorem}. 
We will follow the terminology and ideas presented in \cite[Sections 2 and 3.1]{hodkinsonotto}. 

Let a finite relational language $\mathcal{L}$ be fixed. We start with any finite $\mathcal{L}$-structure $A$, obtain 
an EPPA-extension $B$ of $A$, say by Proposition~\ref{P:nofr}. Of course, at this point, Gaifman clique faithfulness may fail; there may be cliques in $B$ that cannot be sent to $A$ by an automorphism of $B$. Using $B$ we construct a structure $C$ 
extending $A$ which preserves EPPA and in which all such cliques are destroyed.


We now present the details. Let $B\supseteq A$ be a coherent EPPA-extension guaranteed by Proposition~\ref{P:nofr} above. If $A=B$ we are done, so suppose that $A\neq B$. A subset $u\subseteq B$ is called \textit{large} if there is no $g\in {\rm Aut}(B)$ such that $g(u)\subseteq A$.  Otherwise, the subset $u$ is called \textit{small}. Define,
\[
\mathcal{U}=\{u\subseteq B \mid u \text{ is large}\}.
\]
Notice that false cliques and the domain of $B$ are large sets, and the image of a large set under an automorphism of $B$ is also large. The definition of $\mathcal{U}$ comes from \cite[Section 3.1]{hodkinsonotto}. 

Given a finite set $X$, by $[X]$ we denote the set $\{0,1,2,\ldots,|X|-1\}\subseteq \mathbb{N}$.
We recall the following construction from \cite[Section 2.1]{hodkinsonotto}. Let $A\subseteq B$ be as above, and $b\in B$. A map $\chi:\mathcal{U}\to [B]$ is called a \textit{$b$-valuation} provided, for all $u\in\mathcal{U}$, 
$\chi(u)=0$ if $b\notin u$, and $1\leq\chi(u)<|u|$ if $b\in u$.

The domain of the extension $C$ of $A$ is 
\[
C=\big\{(b,\chi) \mid b\in B, ~\chi \text{ is a } b\text{-valuation}\big\}.
\]
We will often write $(b,\chi_b)\in C$, by which we mean that $b\in B$ and $\chi_b$ is \textit{some} $b$-valuation; for the same $b\in B$, there will, in general, be many different $b$-valuations denoted by $\chi_b$.

The following notion comes from \cite[Section 2.1]{hodkinsonotto}. A subset $S\subseteq C$ is called \textit{generic} if for any two distinct points $(b,\chi_b), (c,\chi_c)\in S$:
\begin{enumerate}[label=(\roman*), noitemsep,topsep=0pt]
\item $b\neq c$, and
\item for all $u\in\mathcal{U}$, if both $b, c\in u$, then $\chi_b(u)\neq \chi_c(u)$.
\end{enumerate}

Note that any subset of a generic set is also generic. Define the projection map $\pi:C\to B$ by setting 
\[
\pi(b,\chi)=b.
\]

The following lemma is part of a discussion in \cite[p. 399]{hodkinsonotto}.

\begin{lemma}\label{genericsmallsets}
If $S\subseteq C$ is generic, then $\pi(S)$ is a small subset of $B$.
\end{lemma}

\begin{proof}
Suppose for the sake of contradiction that $u=\pi(S)\subseteq B$ is large. So $u\in \mathcal{U}$. As $S$ is generic, $\pi\restriction_S:S\to u$ is a bijection. We now define a map $\theta:u\to [u]\setminus \{0\}$ by setting $\theta(b)=\chi_b(u)$ for $b\in u$ where $(b, \chi_b)\in S$. 
Again, as $S$ is generic, $\theta$ is injective, but this contradicts that both $u$ and $[u]$ are finite with $|[u]\setminus \{ 0\}|=|u|-1$.
\end{proof}

We now make $C$ into an $\mathcal{L}$-structure in a way that all the $\pi$-fibres in $C$ of large subsets of $B$ are forbidden from being cliques in $C$. This is where all false cliques are destroyed. 

For every $n$-ary relation symbol $R\in\mathcal{L}$ and $n$-tuple $\big((b_1,\chi_1),\ldots,(b_n,\chi_n)\big) \in C$, define $C\models R\big((b_1,\chi_1),\ldots,(b_n,\chi_n)\big)$ if and only if
\begin{enumerate}[label=(\roman*), noitemsep,topsep=0pt]
\item $\{(b_1,\chi_1),\ldots,(b_n,\chi_n)\}$ is a generic subset of $C$, and
\item $B\models R(b_1,\ldots, b_n)$.
\end{enumerate}

Observe that under this definition if $Q\subseteq C$ is a Gaifman clique, then $Q$ is a generic subset of $C$. Therefore, the projection of a Gaifman clique in $C$ is a small subset of $B$. 

\textit{From this point onward in this section, the structures $A, B,$ and $C$ as above are fixed.}

We include the proof of the following fact for the convenience of the reader. See \cite[Lemma 22(ii)]{hodkinsonotto}.

\begin{lemma}\label{L:bemb} The structure $A$ embeds in $C$ with the image of the embedding being a generic subset of $C$.
\end{lemma}
\begin{proof}
We define an embedding $\nu:A\to C$ as follows. Note first that if $u\in\mathcal{U}$ is large, then $u$ is not a subset of $A$, so $|u\cap A|<|u|$. 
For each $u\in\mathcal{U}$, fix an enumeration of $$u\cap A=\{a^u_1,a^u_2,\ldots,a^u_{n}\}$$ where $n<|u|$. Now for each $a\in A$ we define an $a$-valuation $\chi_a:\mathcal{U}\to\mathbb{N}$.

\begin{displaymath}
   \chi_a(u) = \left\{
     \begin{array}{lr}
       0, &  \text{ if }a\notin u,\\
       i \text{ such that } a=a^u_i,  &  \text{ if }a\in u.
     \end{array}
   \right.
\end{displaymath}

For each $a\in A$ we define $\nu(a)=(a,\chi_a)$. The set $\nu(A)$ is a generic subset of $C$, and it follows that $\nu:A\to C$ is an $\mathcal{L}$-embedding.
\end{proof}

Below we denote by $A$ both the original structure $A\subseteq B$ and its image under the embedding from Lemma~\ref{L:bemb}.


Let $p\in \Part(C)$ be a partial automorphism of $C$, and let $g\in {\rm Aut}(B)$. We say that $p$ is \textit{$g$-compatible} if $\pi \circ p\subseteq g \circ \pi$, that is, for all $(b,\chi)\in \dom(p)$ there is a $g(b)$-valuation $\chi'$ such that $p(b,\chi)=\big(g(b),\chi'\big)$.

We use the freedom of choice given in \cite[Lemma 21]{hodkinsonotto} to make additional constraints in constructing the extension $\hat{p}$ of the lemma below. Such constraints will be needed later to make the extension procedure of partial automorphisms coherent.

Let $g\in \Aut(B)$, and let $p\in \Part(C)$ be $g$-compatible partial automorphism with generic domain and range. 
By genericity, for each $b\in B$ there is at most one $\chi_b$ with 
$(b,\chi_b)\in {\rm dom}(p)$. Similarly, for each $b'\in B$ there is at most one $\chi_{b'}$ with $(b',\chi_{b'})\in {\rm range}(p)$. Moreover, the functions 
\[
\pi({\rm dom}(p))\ni b\to \chi_b(u)\in [u]\;\hbox{ and }\; 
\pi({\rm range}(p))\ni b'\to \chi_{b'}(u')\in [u']
\]
are injective. Fixing $u\in {\mathcal U}$ and noticing that $p(b, \chi_b)= (g(b), \chi_{g(b)})$, this observation allows us to define a partial injection from $[u]$ to $[g(u)]=[u]$ by letting 
\begin{equation}\label{E:theta}
\theta^p_u\big(\chi_b(u)\big)=\chi_{g(b)}\big(g(u)\big).
\end{equation}
We extend $\theta^p_u$ to a total permutation of the set $[u]$, fixing $0$, by mapping the set $[u]\setminus\big\{\chi_b(u)\mid (b,\chi_b)\in \dom(p)\big\}$ onto $[u]\setminus\big\{\chi_{g(b)}(g(u))\mid (b,\chi_b)\in \dom(p)\big\}$ in an order-preserving manner. 

Define a map $\hat{p}$ on $C$ by setting for each $(c,\chi_c)\in C$ 
\begin{equation}\label{E:phat}
\hat{p}(c,\chi_c)=\big(g(c),\chi_{g(c)}\big)
\end{equation}
where $\chi_{g(c)}$ is a $g(c)$-valuation given by 
\[
\chi_{g(c)}\big(g(u)\big)=\theta^p_u\big(\chi_c(u)\big) \text{ for } u\in\mathcal{U}.
\]

The following lemma follows from the discussion in \cite[Section 3.1]{hodkinsonotto}.  
\begin{lemma}\label{gcompatibleextension}
Suppose that $g\in \Aut(B)$, and  let $p\in \Part(C)$ be $g$-compatible with generic domain and range. 
Then $\hat{p}$ is $g$-compatible, extends $p$, and belongs to $\Aut(C)$.
\end{lemma}

Using Lemma \ref{gcompatibleextension}, we define a map 
$$\phi:\Part(A) \to \Aut(C)$$
as follows. Let $p\in \Part(A)\subseteq \Part(C)$. By Proposition~\ref{P:nofr}, the partial automorphism $p$ has an extension $g\in \Aut(B)$, and clearly $p$
is $g$-compatible. As $A$ is a generic subset of $C$, both $\dom(p)$, $\range(p)\subseteq A$ are also generic. Now use formula \eqref{E:phat} and apply Lemma \ref{gcompatibleextension} to get a $g$-compatible extension $\hat{p}\in \Aut(C)$ of $p$. Put $\phi(p)=\hat{p}$.
It is in the proof of the next lemma where we really use that $B$ is a \textit{coherent} EPPA-extension of $A$ as given by Proposition~\ref{P:nofr} above.

\begin{lemma}\label{coherence}
The map $\phi: \Part(A) \to \Aut(C)$ is coherent.
\end{lemma}

\begin{proof}
We will show that the image of a coherent triple in $\Part(A)$ under the map $\phi$ is a coherent triple in $\Aut(C)$. Suppose that $p_1, p_2, q\in \Part(A)$ is a coherent triple, that is, $\dom(p_2)=\dom(q)$, $\range(p_2)=\dom(p_1)$,
$\range(p_1)=\range(q)$, and $q=p_1\circ p_2$.
Recall that $A$ is a substructure of both $B$ and $C$. By Proposition~\ref{P:nofr} there are $g_2, g_1, h\in \Aut(B)$
extending $p_2,p_1,q$, respectively, with $h=g_1\circ g_2$. Notice that $p_2$
is $g_2$-compatible, $p_1$ is $g_1$-compatible, and $q$ is $h$-compatible. 
We need to show that $\hat{q}=\hat{p}_1 \circ \hat{p}_2$.

Let $(b,\chi_b), (c,\chi_c)\in C$ be any two points. Here $b,c \in B$, and $\chi_b$ is some $b$-valuation, and $\chi_c$ is some $c$-valuation. By
the construction of $\hat{p}_2$ and $\hat{p}_1$ we get that,
$$\hat{p}_2(b,\chi_b)=\big(g_2(b),\chi_{g_2(b)}\big), \text{ where }\chi_{g_2(b)}\big(g_2(u)\big)=\theta^{p_2}_u\big(\chi_b(u)\big) \text{ for } u\in\mathcal{U},$$
and
$$\hat{p}_1(c,\chi_c)=\big(g_1(c),\chi_{g_1(c)}\big) \text{ where }\chi_{g_1(c)}\big(g_1(v)\big)=\theta^{p_1}_v\big(\chi_c(v)\big) \text{ for } v\in\mathcal{U}.$$

On the one hand, we will compute the value of $\hat{p}_1\big(\hat{p}_2(b,\chi_b)\big)$. Using the above by taking $c=g_2(b)$, $\chi_c=\chi_{g_2(b)}$ and $v=g_2(u)$ we get 
$$\hat{p}_1\big(\hat{p}_2(b,\chi_b)\big)=\hat{p_1}\big(g_2(b),\chi_{g_2(b)}\big)=\big(g_1g_2(b),\chi_{g_1(g_2(b))}\big)=\big(h(b),\chi_{h(b)}\big)$$
where $\chi_{h(b)}$ has the following value for each $u\in \mathcal{U}$,
$$\chi_{h(b)}\big(h(u)\big)=\chi_{g_1(g_2(b))}\big(g_1g_2(u)\big)=\theta^{p_1}_{g_2(u)}\Big(\chi_{g_2(b)}\big(g_2(u)\big)\Big)= \theta^{p_1}_{g_2(u)} \circ \theta^{p_2}_u\big(\chi_b(u)\big).$$

On the other hand, we have that
$$\hat{q}(b,\chi_b)=\big(h(b),\psi_{h(b)}\big) \text{ where }$$
$$\psi_{h(b)}\big(h(u)\big)=\theta^{q}_u\big(\chi_b(u)\big) \text{ for } u\in\mathcal{U}.$$
Therefore, we reach our desired result if we show that $\chi_{h(b)}= \psi_{h(b)}$, which follows from showing that 
\begin{equation}\label{E:compo}
\theta^{q}_u\big(\chi_b(u)\big)=\theta^{p_1}_{g_2(u)} \circ \theta^{p_2}_u\big(\chi_b(u)\big)
\end{equation}
for each $(b,\chi_b)\in C$ and $u\in\mathcal{U}$. 

In order to prove \eqref{E:compo}, we fix $(b,\chi_b)\in C$ and $u\in\mathcal{U}$, and let
\[
m=\chi_b(u).
\]

Recall that $\dom(p_2)=\dom(q)$, $\range(p_2)=\dom(p_1)$, $\range(p_1)=\range(q)$ are all generic sets as they are subsets of the generic set $A\subseteq C$, and so we can write their elements in the form $(c,\chi_c)$ without ambiguity, where $\chi_c$ is some $c$-valuation.

\textbf{Case 1.} $m=\chi_c(u)$ for some $(c,\chi_c)\in \dom(p_2)=\dom(q)$.

We have 
$$p_1 \circ p_2(c,\chi_c)=p_1\big(g_2(c),\chi_{g_2(c)}\big)=\big(g_1g_2(c),\chi_{g_1g_2(c)}\big)=\big(h(c),\chi_{h(c)}\big).$$ 
Using this information and \eqref{E:theta}, we get 
$$\theta^{p_1}_{g_2(u)} \circ \theta^{p_2}_u(m) =\theta^{p_1}_{g_2(u)} \circ
\theta^{p_2}_u(\chi_c(u))=\theta^{p_1}_{g_2(u)}\big(\chi_{g_2(c)}(g_2(u))\big) = \chi_{h(c)}(h(u)).
$$
As $q=p_1 \circ p_2$ we have that $q(c,\chi_c)=\big(h(c),\chi_{h(c)}\big)$, and so by construction of $\theta^{q}_u$
we get,
$$\theta^{q}_u(m)=\theta^{q}_u(\chi_c(u))=\chi_{h(c)}(h(u)).$$

Therefore, $\theta^{p_1}_{g_2(u)} \circ \theta^{p_2}_u(m)=\theta^{q}_u(m)$.

\textbf{Case 2.} $m\neq\chi_c(u)$ for all $(c,\chi_c)\in \dom(p_2)=\dom(q)$.

Suppose that $m$ is the $i^{th}$
element of $[u]$ that is not of the form $\chi_c(u)$ for some $(c,\chi_c)\in \dom(p_2)$. Then $\theta^{p_2}_u(m)$ is the $i^{th}$ element of $[u]$ not of the form 
$\chi_{g_2(c)}(g_2(u))$ for $(c,\chi_c)\in \dom(p_2)$.  Note that $[g_2(u)]=[u]$ as $g_2$ is a bijection, meaning that
$\theta^{p_1}_{g_2(u)}$ is also a permutation of the set $[u]$. Finally, as $\range(p_2)=\dom(p_1)$ we get that
$k=\theta^{p_1}_{g_2(u)} \circ \theta^{p_2}_u(m)$ is the $i^{th}$ element of $[u]$ such that $k\neq
\chi_{h(c)}(h(u))$ for $(c,\chi_c)\in \dom(p_2)$.

Now, by construction of $\theta^{q}_u$ and as $\dom(p_2)=\dom(q)$, we have that $\theta^q_u(m)$ is the
$i^{th}$ element of $[u]$ not of the form $\chi_{h(c)}(h(u))$ for  $(c,\chi_c)\in \dom(q)$. Thus, $k=\theta^q_u(m)$, and so
$\theta^{p_1}_{g_2(u)}\circ \theta^{p_2}_u(m)=\theta^{q}_u(m)$.

Therefore, we have shown that $\theta^{q}_u\big(\chi_b(u)\big)=\theta^{p_1}_{g_2(u)} \circ \theta^{p_2}_u\big(\chi_b(u)\big)$ for any
$(b,\chi_b)\in C$ and $u\in\mathcal{U}$, implying that $\chi_{h(b)}=\psi_{h(b)}$ and so we get that $\hat{p}_1 \circ \hat{p}_2=\hat{q}$. So the map $p\mapsto \hat{p}$ from $\Part(A)$ to $\Aut(C)$ is coherent.
\end{proof}

We now give the proof of Theorem \ref{HOstrengthenedtheorem}. 

\textbf{Proof of Theorem \ref{HOstrengthenedtheorem}}~
Let $A$ be a finite $\mathcal{L}$-structure. By Proposition~\ref{P:nofr}, there is an extension $B$ of $A$ in which every
element of $\Part(A)$ extends to an element of $\Aut(B)$ such that the corresponding map is coherent. From $B$ construct the
$\mathcal{L}$-structure $C=\big\{(b,\chi_b) \mid b\in B, ~\chi_b \text{ is a } b\text{-valuation}\big\}$ as described above in this section. Every element of $\Part(A)$ extends to $\hat{p}\in \Aut(C)$ given by the map $\phi$ above. By Lemma \ref{coherence}, this $\phi:\Part(A) \to \Aut(C)$ is coherent. Finally, by \cite[Section 3.1]{hodkinsonotto} every clique in $C$ is the image of a clique in $A$ under an automorphism of $C$. To see this, let $Q\subseteq C$ be a Gaifman clique. Then $Q$ is a generic subset by definition of the structure on $C$. So by Lemma \ref{genericsmallsets}, $\pi(Q)\subseteq B$ is a small subset. Hence, there is some $g\in \Aut(B)$ such that $Q'=g(\pi(Q))\subseteq A$. We finish by applying Lemma \ref{gcompatibleextension} to the partial automorphism $p=g\pi:Q\to Q'$ which is $g$-compatible with generic domain and range.  \hfill $\blacksquare$


\section {Free Amalgamation Classes and Coherent EPPA}\label{freeamalgamationsection}

In this section we present the notion of free amalgamation, and clarify the relationship between free amalgamation classes and classes of structures which forbid a family of Gaifman cliques. This allows us to apply Theorem \ref{forbclasses}(ii) and conclude that free amalgamation classes have coherent EPPA.   

\begin{definition}\label{freeamalgamdef}\rm
Let $\mathcal{L}$ be a relational language. Given finite $\mathcal{L}$-structures $A, B_1, B_2$ with $A\subseteq B_1$ and $A\subseteq B_2$, the \textit{free amalgam} of $B_1$ and $B_2$ over $A$ is the structure $C$ whose domain is the disjoint union of $B_1$ and $B_2$ over $A$, and for every relation symbol $R\in\mathcal{L}$ we define $R^C:=R^{B_1} \cup R^{B_2}$.
\end{definition}

We have the following two observations on the free amalgam $C$. First, when $B_1, B_2$ are viewed as subsets of $C$ we have that $B_1 \cap B_2 = A$. Second, there is no relation symbol $R\in \mathcal{L}$ and a tuple $\bar{c}\in C$ such that $\bar{c}$ meets both $B_1\setminus A$ and $B_2\setminus A$, and $C\models R(\bar{c})$.

\begin{definition}\rm Let $\mathcal{L}$ be a relational language, and $\mathcal{C}$ be a class of finite $\mathcal{L}$-structures.  
\begin{itemize} [noitemsep,topsep=-3pt,parsep=2pt,partopsep=0pt]
 \item The class $\mathcal{C}$ has the \textit{free amalgamation property} if $\mathcal{C}$ is closed under taking free amalgams.
 
 \item The class $\mathcal{C}$ is called a \textit{free amalgamation class} if it is a Fra\"{i}ss\'{e} class with the free amalgamation property.
 
 \item The Fra\"{i}ss\'{e} limit of a free amalgamation class is called a \textit{free homogeneous structure}.
\end{itemize}
\end{definition}

Note that the free amalgamation property implies the amalgamation property.   

\begin{example}\rm
The following are examples of free homogeneous structures.

\begin{enumerate} [noitemsep,topsep=-3pt,parsep=2pt,partopsep=0pt]
\item The random graph \cite{cameronrandomgraph}.
\item The universal homogeneous $K_n$-free graph \cite[Example 2.2.2]{macphersonsurvey}.
\item The universal homogeneous directed graph.
\item The continuum many Henson directed graphs \cite{henson}.
\item The universal homogeneous $k$-hypergraph \cite{thomas}.
\item The universal homogeneous tetrahedron-free $3$-hypergraph, where a tetrahedron is a complete $3$-hypergraph on four vertices.
\item The Fra\"{i}ss\'{e} limit of the class of all finite $3$-hypergraphs such that every subset of size $4$ contains at most two $3$-hyperedges.
\end{enumerate}
\end{example}

In situations where we have a binary relation which is either transitive or total, one expects free amalgamation to fail. For example the classes of all finite partial orders, linear orders, tournaments, and structures with an equivalence relation do not have the free amalgamation property. Another example of a Fra\"{i}ss\'{e} class without free amalgamation is the class of all finite \textit{two-graphs}, where a two-graph is a $3$-hypergraph such that every subset of size $4$ has an even number of hyperedges---see \cite[Example 2.3.1.4]{macphersonsurvey}.

\begin{definition}\rm
Let $\mathcal{C}$ be a class of finite $\mathcal{L}$-structures. A finite $\mathcal{L}$-structure $F$ is called \textit{minimal forbidden} in $\mathcal{C}$ if $F\notin \mathcal{C}$ and for any $v\in F$ we have that $F\setminus \{v\}$ is in $\mathcal{C}$.
\end{definition}

One can observe that if $F$ is a finite $\mathcal{L}$-structure such that $F\notin \mathcal{C}$, then $F$ contains a minimal forbidden substructure. For if $F$ were not a minimal forbidden structure, there is a vertex $v\in F$, such that $F\setminus \{v\}$ is still not in $\mathcal{C}$. We keep repeating this process until we find a substructure $F_0 \subseteq F$ which is minimal forbidden.

Recall that $\Forb_e(\mathcal{F})$ is the class of all finite $\mathcal{L}$-structures which are $\mathcal{F}$-free under embeddings. The class $\Forb_e(\mathcal{F})$ has the hereditary property. Conversely, suppose that $\mathcal{C}$ is a class of finite $\mathcal{L}$-structures closed under isomorphism and having the hereditary property. Let $\mathcal{F}$ be the family of all finite structures which are minimal forbidden in $\mathcal{C}$. Then $\mathcal{C}=\Forb_e(\mathcal{F})$. To see this, first suppose that $A\in \Forb_e(\mathcal{F})$ but $A\notin \mathcal{C}$. Then $A$ contains some element in $\mathcal{F}$, but this contradicts $A$ is $\mathcal{F}$-free under embeddings. So $\Forb_e(\mathcal{F})\subseteq \mathcal{C}$. For the other direction, supposing that $A\in\mathcal{C}$ but $A\notin\Forb_e(\mathcal{F})$, there is some $F\in \mathcal{F}$ and an embedding $g:F\to A$. As $\mathcal{C}$ has the hereditary property, we get $F \in\mathcal{C}$, contradicting $F \notin \mathcal{C}$. So $\mathcal{C}\subseteq \Forb_e(\mathcal{F})$.

The following lemma seems to be part of the folklore; we learned about it from Dugald Macpherson. 

\begin{lemma}\label{freeamalgamgaifmancliques}
Let $\mathcal{L}$ be a relational language, and $\mathcal{C}$ be a class of finite $\mathcal{L}$-structures. Then $\mathcal{C}$ is a free amalgamation class if and only if $\mathcal{C}=\Forb_e(\mathcal{F})$ for some family $\mathcal{F}$ of Gaifman cliques.
\end{lemma}
\begin{proof}
Let $\mathcal{C}$ be a free amalgamation class. By the above, $\mathcal{C}=\Forb_e(\mathcal{F})$ where $\mathcal{F}$ is the family of all the minimal forbidden structures in $\mathcal{C}$. It remains to show that every element $Q\in \mathcal{F}$ is a Gaifman clique. If not, then there is some $Q\in \mathcal{F}$ containing two distinct elements $u,v \in Q$ which do not satisfy any relation of $\mathcal{L}$. Let $Q_u=Q\setminus\{u\}$ and $Q_v=Q\setminus \{v\}$. By minimality of $Q$, both $Q_u$ and $Q_v$ belong to $\mathcal{C}$. Moreover, $Q_{uv}:=Q\setminus\{u,v\}$ belongs to $\mathcal{C}$ by the hereditary property. By the free amalgamation property of $\mathcal{C}$, we get that $Q$ which is the free amalgam of $Q_u$ and $Q_v$ over $Q_{uv}$ is in $\mathcal{C}$, contradicting $Q\in \mathcal{F}$. Therefore, every $Q\in\mathcal{F}$ is a Gaifman clique.

For the reverse direction, suppose that $\mathcal{C}=\Forb_e(\mathcal{F})$ for some collection $\mathcal{F}$ of Gaifman cliques. Let $A, B_1, B_2 \in \mathcal{C}$ such that $A\subseteq B_1$ and $A\subseteq B_2$. Let $C$ be the free amalgam of $B_1$ and $B_2$ over $A$. We claim that $C\in\mathcal{C}$. If $C$ were not in $\mathcal{C}$, then there is a Gaifman clique $Q\in\mathcal{F}$ and embedding $g:Q\to C$. Moreover, there are two vertices $u,v\in Q$ with $u\in B_1\setminus A$ and $v\in B_2 \setminus A$. But $u$ and $v$ are related by some $R\in \mathcal{L}$, contradicting $C$ a free amalgam.
\end{proof}

So Lemma \ref{freeamalgamgaifmancliques} together with Theorem \ref{forbclasses}(ii) give the following corollary.

\begin{corollary} \label{freeamalgamationEPPA}
Suppose that $\mathcal{L}$ is a finite relational language. Then any free amalgamation class of finite $\mathcal{L}$-structures has coherent EPPA.
\end{corollary}

We give below an example of a free amalgamation class which cannot be written as a class which forbids a family of structures under homomorphisms, rather than embeddings. So Herwig-Lascar Theorem \cite[Theorem 3.2]{herwiglascar} could not be applied in this situation. However, by Corollary \ref{freeamalgamationEPPA} this class has coherent EPPA.  

\begin{example}\rm  Let $\mathcal{L}$ be the language of $3$-hypergraphs, that is, $\mathcal{L}$ contains one ternary relation symbol $R$. A $3$-hypergraph is an $\mathcal{L}$-structure such that $R$ is interpreted as an irreflexive symmetric ternary relation. A $3$-tuple which satisfies $R$ is called a hyperedge. Let $Q$ be a $3$-hypergraph on four vertices with exactly $3$ hyperedges. Let $\mathcal{C}$ be the class of all finite $3$-hypergraphs which forbid $Q$ under embeddings. The class $\mathcal{C}$ is a free amalgamation class, and so has EPPA by Corollary \ref{freeamalgamationEPPA} above. Recall that a tetrahedron $H$ is a complete $3$-hypergraph on four vertices, and note that $H \in \mathcal{C}$.
Now suppose that there is a family $\mathcal{F}$ of $\mathcal{L}$-structures such that $\mathcal{C}=\Forb_h(\mathcal{F})$. Then as $Q\notin\mathcal{C}$, there is $F\in\mathcal{F}$ and a homomorphism $h:F\to Q$. Let $\alpha:Q\to H$ be any bijective  map. Then $\alpha$ is a homomorphism, and so $\alpha h:F \to H$ is a homomorphism too. So $H$ is not $\mathcal{F}$-free under homomorphisms, contradicting that $H\in \mathcal{C}$.
\end{example}


\section{The Automorphism Group of a Fra\"{i}ss\'{e} limit}\label{automorphismgroupsection}

Let $M$ be a countably infinite $\mathcal{L}$-structure, and put $G:=\Aut(M)$. The pointwise stabiliser of a subset $A\subseteq M$ is denoted by $G_A$, and the orbit of an element $a\in M$ under the action of $G$ is denoted by $a^G$. The automorphism group $\Aut(M)$ is endowed with the pointwise convergence topology whose basis consists of all cosets of pointwise stabilisers of finite subsets of $M$. That is, a basic open set has the form: $$[p]:=\{g\in\Aut(M) \mid p\subseteq g\}$$ where $p:A\to B$ is a finite partial automorphism of $M$. Note that $[p]=hG_A$ for some $h\in\Aut(M)$ such that $p\subseteq h$. With this topology $\Aut(M)$ is a Polish group, that is, a separable completely metrisable topological group. 

In this section we focus on the case when $M$ is a Fra\"{i}ss\'{e} limit, that is, a homogeneous structure. We show that if $\Age(M)$ has coherent EPPA, then $\Aut(M)$ contains a dense locally finite subgroup. We also show that if $M$ is a free homogeneous structure, then $\Aut(M)$ admits ample generics.

\subsection{A Dense Locally Finite Subgroup}\label{DLFsubgroupsection}

With respect to the pointwise convergence topology, a subgroup $H\leq \Aut(M)$ is dense if and only if for any $g\in \Aut(M)$ and finite $A\subseteq M$ there is $h\in H$ such that $g(a)=h(a)$ for all $a\in A$, that is, $H$ has the same orbits as $\Aut(M)$ in $M^n$ for all $n\in \omega$.  

\begin{proposition}\label{coherentEPPADLFsubgrouptheorem}
Suppose that $M$ is a homogeneous relational structure such that $\Age(M)$ has coherent EPPA. Then $\Aut(M)$ contains a dense locally finite subgroup.
\end{proposition}

\begin{proof}
We will build a chain $A_0\subseteq A_1 \subseteq \ldots\subseteq A_i\subseteq A_{i+1}\subseteq \ldots~$ of finite substructures of $M$ such that $M=\bigcup_{i\in\omega}A_i$, and simultaneously we build a directed system $H_0 \to \ldots \to H_i \xrightarrow{\phi_i} H_{i+1} \to\ldots~$ of finite groups such that for each $i\in\omega$ we have that $H_i\leq {\rm Aut}(A_i)$, and the map $\phi_i:H_i\to H_{i+1}$ is a group embedding such that $\phi_i(h)$ extends $h$ for every $h\in H_i$. Then, the dense locally finite subgroup of $\Aut(M)$ will be $H=\varinjlim H_i$, the direct limit of the directed sequence $(H_i)_{i\in\omega}$.

Enumerate all finite partial automorphisms of $M$ as
$\big\{p_i:U_i\to V_i \mid i \in \omega\big\}$.
Here $U_i, V_i$ are finite substructures of $M$, and $p_i$ is an isomorphism. Choose some $a\in M$, and start by putting $A_0=\{a\}$ and $H_0={\rm Aut}(A_0)$. Suppose that stage $i$ has been completed and we have a finite substructure $A_i\subseteq M$ and a group $H_i\leq \Aut(A_i)$. We will proceed to construct stage $i+1$.

We will ensure that $U_i \cup V_i \subseteq A_{i+1}$ and that $H_{i+1}$ contains an element extending $p_i$. We apply coherent EPPA to the substructure $B:=A_i \cup U_i \cup V_i$ to obtain a structure $A_{i+1}\in \Age(M)$ with $B\subseteq A_{i+1}$, and a coherent map $\phi:\Part(B) \to \Aut(A_{i+1})$ witnessing EPPA. By homogeneity of $M$ we may assume that $A_i\subseteq A_{i+1}\subseteq M$. 

The partial automorphism $p_i:U_i \to V_i$ belongs to $\Part(B)$ and so it extends to $\phi(p_i)\in \Aut(A_{i+1})$. Finish by putting: 
$$H_{i+1}:=\langle \phi\big(H_i \cup \{p_i\}\big) \rangle \leq \Aut(A_{i+1}) \ \text{ and }\ \phi_i:=\phi\restriction_{H_i}.$$
Note that $\phi_i: H_i \to H_{i+1}$ is a group embedding such that $h \subseteq \phi_i(h)$ for every $h\in H_i$. 

Clearly $M=\bigcup_{i\in\omega}A_i$. The construction ensures that the group $H:=\varinjlim H_i$ intersects every nonempty basic open subset of $\Aut(M)$ since every such subset is of the form $[p_i]$ for some $p_i$ in the enumeration above. Finally, the finiteness of each $H_i$ implies that $H$ is locally finite.
\end{proof} 

Since free amalgamation classes have coherent EPPA by Corollary \ref{freeamalgamationEPPA}, we have the following result which generalises \cite[Theorem 1.1]{bhattacharjeemacpherson} of Bhattacharjee and Macpherson who proved that the automorphism group of the random graph has a dense locally finite subgroup. 

\begin{corollary}\label{DLFsubgrouptheorem}
Let $M$ be a free homogeneous structure over a finite relational language. Then $\Aut(M)$ contains a dense locally finite subgroup.
\end{corollary}

We have seen above that coherent EPPA leads to the existence of a dense locally finite subgroup, the following lemma treats the reverse direction. See \cite[Proposition 6.4]{kechrisrosendal} for a more general statement.

\begin{proposition}\label{DLFsubgroupimpliesEPPA}
Let $M$ be a homogeneous relational structure. Suppose that ${\rm Aut}(M)$ has a dense locally finite subgroup. Then $\Age(M)$ has EPPA.
\end{proposition}

\begin{proof}
Let $H\leq {\rm Aut}(M)$ be a dense locally finite subgroup. Fix some $A\in {\rm Age}(M)$, and assume that $A\subseteq M$. Enumerate $\Part(A)=\{p_1,\ldots,p_n\}$. By the homogeneity of $M$ there are $f_1,\ldots,f_n \in {\rm Aut}(M)$ such that $p_i\subseteq f_i$. As $H$ is dense, we may assume that each $f_i\in H$. As $H$ is locally finite, the subgroup $F:=\langle f_1,\ldots,f_n\rangle\leq H$ is finite. Define the finite substructure $B:=\bigcup\{f(A)\mid f\in F\}$ of $M$. Clearly $f(B)=B$ for all $f\in F$, and so each $f_i\restriction_B \in \Aut(B)$ and extends $p_i$. Thus, $B$ is an EPPA-extension of $A$.
\end{proof}

\textbf{Question.} Is it possible to obtain coherent EPPA, rather than just EPPA, in the conclusion of Proposition \ref{DLFsubgroupimpliesEPPA} above?

\subsection{Ample Generics}\label{amplegenerics}

We now proceed towards the existence of ample generics for free homogeneous structures. The action of diagonal conjugation of a group $G$ on $G^n$ is given by
$g\cdot (h_1,\ldots,h_n)=(gh_1g^{-1},\ldots,gh_ng^{-1})$. The following notion originates from \cite{HHLS}, and the version below is from \cite{kechrisrosendal}.

\begin{definition}\rm
A Polish group $G$ has \textit{ample generics} if  for each $n\geq 1$, $G$ has a comeagre orbit in its action on $G^n$ by diagonal conjugation.
\end{definition}

We say that a countably infinite structure $M$ has ample generics if ${\rm Aut}(M)$ has ample generics. We discuss briefly a consequence of the existence of ample generics. A subgroup $H\leq G=\Aut(M)$ has \textit{small index} if $|G:H|<2^{\aleph_0}$. We say that $M$ has the \textit{small index property} if any subgroup of $\Aut(M)$ of small index is open. So when $M$ has the small index property, the topological structure of the ${\rm Aut}(M)$ is determined by its abstract group structure, as a subgroup of ${\rm Aut}(M)$  is open precisely if it has a small index. An interesting result of \cite{HHLS} and \cite[Theorem 1.6]{kechrisrosendal} is that if $M$ has ample generics, then $M$ has the small index property. It was shown in \cite{HHLS} that the random graph has ample generics, and that the automorphism group of any $\omega$-stable, $\omega$-categorical structure contains an open subgroup with ample generics. 

In our situation, the methods of \cite{HHLS} can be used to establish the existence of ample generics for free homogeneous structures.

\begin{theorem}\label{amplegenericstheorem} 
Any free homogeneous structure over a finite relational language has ample generics.
\end{theorem}

\begin{proof}
Let $M$ be a free homogeneous structure over a finite relational language. By Corollary \ref{freeamalgamationEPPA}, ${\rm Age}(M)$ has EPPA. Consequently, it follows that the set $\Gamma$ of all finite subsets of $M$ is an amalgamation base of $M$ (see \cite[Definition 2]{ivanovstrongly} and \cite[Definition 2.8]{HHLS}). 
To see this, let $p_1, \ldots, p_n$ be finite partial automorphisms of $M$. Put $A=\bigcup_{i=1}^n(\dom(p_i)\cup \range(p_i))$, then by EPPA there is an extension $B\in \Age(M)$ of $A$ such that each $p_i$ extends to an automorphism of $B$. By homogeneity of $M$, we may assume that $A\subseteq B \subseteq M$. This establishes the first condition of an amalgamation base. To show the second condition, let $A, B, C$ be in $\Gamma$ such that $A\subseteq B$ and $A\subseteq C$. Let $D\in \Age(M)$ be the free amalgam of $B$ and $C$ over $A$. Again by homogeneity of $M$, we may assume that $C\subseteq D \subseteq M$. Let $B'\subseteq D$ be the isomorphic copy of $B$. Then if $\alpha \in \Aut(B')$ and $\beta \in \Aut(C)$ such that $\alpha\restriction_A = \beta\restriction_A$, then their union $\alpha \cup \beta$ is an automorphism of $D$, this holds because of the way the structure on $D$ was defined. We finish by invoking \cite[Proposition 3]{ivanovstrongly} to conclude that $M$ has ample generics.      
\end{proof}

Any homogeneous structure over a finite relational language is $\omega$-categorical. So from \cite[Theorems 6.9, 6.12, and 6.19, and Corollary 1.9]{kechrisrosendal} we obtain the following.

\begin{corollary}
Suppose that $M$ is a free homogeneous structure over a finite relational language. Then $\Aut(M)$ has the small index property, uncountable cofinality, 21-Bergman property, and Serre's property (FA).
\end{corollary}

\textbf{Acknowledgements.} D. Siniora is extremely thankful to Dugald Macpherson for his support and useful suggestions.

\bibliographystyle{abbrv}
\bibliography{DNSreferences}

\end{document}